\documentclass[journal,onecolumn]{IEEEtran}

\usepackage{amsmath} 
\usepackage{amssymb}  
\usepackage{amsthm} 
\usepackage{amsfonts}
\usepackage{graphicx}
\usepackage{xcolor}
\usepackage{mathtools}
\usepackage{subcaption}
\newcommand*{\Scale}[2][4]{\scalebox{#1}{$#2$}}
\usepackage{enumitem}
\usepackage{lipsum}
\usepackage{mathtools}
\usepackage{cuted}
\usepackage{cite}

\usepackage[hidelinks]{hyperref}

\usepackage{mathrsfs} 
\newtheorem{remark}{Remark}
\newtheorem{theorem}{Theorem}
\newtheorem{lemma}{Lemma}

\usepackage{pgf,tikz}
\usepackage[utf8]{inputenc}
\usepackage{pgfplots} 
\usepackage{pgfgantt}
\usepackage{pdflscape}
\pgfplotsset{compat=newest} 
\pgfplotsset{plot coordinates/math parser=false}

\usepackage{caption}
\usepackage{booktabs}

\usepackage{verbatim}
\def\h{2.8}
\def\l{7.5}

\begin{document}
%
\title{\LARGE \bf Modelling the Effect of Vaccination and Human Behaviour on the Spread of Epidemic Diseases on Temporal Networks}
%
%
%

\author{Kathinka Frieswijk, Lorenzo Zino, and Ming Cao
\thanks{The authors are with the Faculty of Science and Engineering, University of Groningen, The Netherlands, 9747 AG Groningen, e-mail:
       \href{mailto:k.frieswijk@rug.nl}{\texttt{
      $\left\{\text{k.frieswijk,lorenzo.zino,m.cao}\right\}$@rug.nl}}. The work was partially supported by the European Research Council (ERC--CoG--771687).}%
}

\maketitle

\begin{abstract}
\noindent Motivated by the increasing number of COVID-19 cases that have been observed in many countries after the vaccination and relaxation of non-pharmaceutical interventions, we propose a mathematical model on time-varying networks for the spread of recurrent epidemic diseases in a partially vaccinated population. The model encapsulates several realistic features, such as the different effectiveness of the vaccine against transmission and development of severe symptoms, testing practices, the possible implementation of non-pharmaceutical interventions to reduce the transmission, isolation of detected individuals, and human behaviour. Using a mean-field approach, we analytically derive the epidemic threshold of the model and, if the system is above such a threshold, we compute the epidemic prevalence at the endemic equilibrium. These theoretical results show that precautious human behaviour and effective testing practices are key toward avoiding epidemic outbreaks. Interestingly, we found that, in many realistic scenarios, vaccination is successful in mitigating the outbreak by reducing the prevalence of seriously ill patients, but it could be a double-edged sword, whereby in some cases it might favour resurgent outbreaks, calling for higher testing rates, more cautiousness and responsibility among the population, or the reintroduction of non-pharmaceutical interventions to achieve complete eradication.
\end{abstract}

%
\IEEEpeerreviewmaketitle

\section{Introduction}
%
%
%
%
\IEEEPARstart{T}{he} COVID-19 (SARS-CoV-2) pandemic shook the world to its core, spreading quickly, while leaving death, economical distress and despair in its path \cite{who}. To control the spread of the disease, unrivalled efforts have been directed towards the development of effective vaccines \cite{whovaccintracker}. Although COVID-19 vaccines offer protection against the development of severe symptoms and transmission, reinfection with COVID-19 has been documented \cite{european2021risk,dagan2021bnt162b2,cavanaugh2021reduced,townsend2021durability}. Hence, it is unclear whether a high vaccination coverage alone is sufficient to eradicate the global pandemic, as massive outbreaks have been recently registered even in countries with high vaccination rates~\cite{who}, as soon as non-pharmaceutical interventions (NPIs) were relaxed.

To gain insight into the mechanics behind the spread of infections, it is standard practice within the scientific community to employ mathematical models\cite{pastor,7393962,dyncontrolmei,Pare2020review}. In particular, epidemic models on temporal networks have become very popular in the last decade for their ability to capture the complex and time-varying patterns of human encounters, through which epidemic diseases are transmitted~\cite{Pare2018,zinoreview}. The analysis of such models leads to knowledge that can subsequently be used to inform control and intervention strategies \cite{7393962,dyncontrolmei,zinoreview,9089218}. During the past phases of the COVID-19 health crisis, the systems and control community has worked incessantly towards developing mathematical models to predict the spread of the pandemic \cite{giulia,DellaRossa2020,Calafiore2020}, to understand the effectiveness of non-pharmaceutical interventions~\cite{parino2021,Carli2020,Soltesz2020}, and to design vaccination campaigns~\cite{giordano2021vaccination,Parino2021vaccine, grundel2021}. Hence, mathematical modelling can be a key tool for studying the current challenges of the COVID-19 pandemic, specifically those related to the effectiveness of vaccination and the possibility to fully eradicate the disease.



To this aim, in this paper, we propose a network model for the spread of recurrent epidemic diseases, for instance those caused by fast mutating viruses (e.g. influenza viruses) or those that do not provide permanent immunity (e.g. COVID-19). In our model, we use a mechanism inspired by continuous-time activity-driven networks~\cite{zino1,zino2} to generate the time-varying pattern of physical encounters at close promixity, through which the disease is transmitted. The proposed mechanism takes into account the specific human behaviour in the form of a responsibility level, which represents the probability that an individual will choose to protect others and maintain distance to others, when only mild symptoms of the infection are apparent. To incorporate this human behaviour, we add an extra compartment to a standard susceptible--infected--susceptible (SIS) model~\cite{zinoreview}, to account for individuals that are mildly symptomatic and for those that have severe symptoms, and are thus restrained from having social interactions. In our model, we include three control measures which impact the spreading of the virus: i) \emph{vaccination}, where we have distinct factors for capturing its effect on transmission and on developing serious illness; ii) \emph{non-pharmaceutical interventions} (NPIs), such as mandatory face masks and physical distancing; and iii) \emph{testing campaigns}, which aim at identifying mildly symptomatic individuals and isolate them, thus reducing the contagions.

The main contribution of this paper is three-fold. First, we propose a parsimonious yet flexible epidemic model that incorporates human behaviour, the effects of vaccination, testing, and NPIs. Second, we perform a theoretical analysis of the proposed epidemic model, utilising a mean-field approach in the limit of large scale populations~\cite{virusspread}. Our analysis allows to compute the epidemic threshold and the epidemic prevalence if the outbreak trespasses the epidemic threshold and the disease becomes endemic. Third, we discuss a case study calibrated on the COVID-19 pandemic. Our findings elucidate the role of vaccination efforts on the spreading of fast transforming viruses, and suggest that although vaccines are typically highly effective in mitigating an epidemic outbreak by reducing the population with severe symptoms, they may make controlling local outbreaks more challenging. This potential double-edged sword effect may call for an increase of testing efforts and individuals' responsibility in populations with a high vaccination coverage, in order to avoid resurgent outbreaks and the associated implementation of NPIs to keep the disease under control. 

The organisation of the rest of this paper is as follows. The notation used throughout this paper is collected in Section II. In Section III, we formulate the model, and subsequently discuss its dynamics in Section IV. Next, we present our theoretical results in Section V, followed by a case study on COVID-19 in Section VI. The paper is concluded with Section VII, in which we discuss future research avenues.

\section{Notation}
\noindent We gather here the notation used in this paper. The set of strictly positive  integer, real non-negative, and strictly positive real numbers is denoted by $\mathbb{N}$, $\mathbb{R}_{\ge 0}$, and $\mathbb{R}_{> 0}$, respectively. Given a set $\mathcal S$, $|\mathcal S|$ denotes its cardinality. Given a function $x(t)$ with $t \in \mathbb{R}_{\ge 0}$, we define $x(t^+) = \lim_{s \searrow t} x(s)$, and $x(t^-) = \lim_{s \nearrow t} x(s)$.

\section{Model}

\noindent We consider $n$ individuals labelled as $\mathcal{V} = \{1, \hdots, n\}$ who interact on an undirected network $\left(\mathcal{V}, \mathcal{E}(t)\right)$, whose edge set evolves in continuous time $t\in \mathbb{R}_{\ge 0}$. If $(j,k)\in\mathcal{E}(t)$, then a physical encounter in close proximity is occurring between individuals $j$ and $k$ at time $t$. 

\subsection{Disease transmission model}\label{Epidemicmodel}

\noindent We extend the classical susceptible--infected--susceptible (SIS) model~\cite{zinoreview} by dividing infected individuals in two separate compartments for \emph{infectious} (i.e.\ infected individuals who are untested and mildly symptomatic) and \emph{quarantined} (due to the presence of severe symptoms or of a positive test) individuals, respectively. Hence, the health state of individual  $j \in \mathcal{V}$ at time $t\in\mathbb R_{\geq 0}$ is defined by the variable $X_j(t) \in \{\text{S}, \text{I}, \text{Q}\}$, denoting susceptible ($\text{S}$), infectious ($\text{I}$), and quarantined infected individuals ($\text{Q}$). 

Here, we make the assumption that quarantined individuals ($\text{Q}$) do not actively participate in the society because they are too unwell or due to prohibitions enforced by public health authorities. Infectious individuals with mild symptoms ($\text{I}$) can  participate, however. Specifically, we introduce a parameter $\sigma\in [0,1]$ that captures the individuals' level of \emph{responsibility}. If an individual $j$ is mildly symptomatic, they choose to protect others and maintain physical distance with probability $\sigma$; whereas with probability $1-\sigma$, $j$ disregards the symptoms and physically interacts with others in close proximity. We assume that each individual makes their decision independently of the others and of the past. For the sake of simplicity, we have assumed that all the individuals have the same level of responsibility, but heterogeneous responsibilities could easily be introduced in the model.

The time-varying network of physical encounters in close proximity is generated in a stochastic fashion, as detailed in the following. Inspired by continuous-time adaptive activity-driven networks~\cite{zino2}, we attach to each individual a Poisson clock with unit rate\footnote{A Poisson clock with rate $\zeta$ is a continuous-time stochastic process that ticks in a time-interval of length $\Delta t$ with probability $\zeta \Delta t+o(\Delta t)$.}, ticking independently of the others. If the clock associated with individual $j\in\mathcal V$ ticks at time $t\in\mathbb R_{\geq 0}$, then $j$ activates and has an interaction with another individual $k$, selected uniformly at random from $\mathcal{V}\setminus\{i\}$. Whether a susceptible individual has an encounter in close proximity, i.e.\ a physical encounter that allows for the transmission of the infection, depends on the state and responsibility level of the individuals involved in the interaction.
Specifically, if $j$ is susceptible at the moment of activation ($X_j(t^-)=\text{S}$) and makes a connection with an infected and mildly symptomatic individual ($X_k(t^-)=\text{I}$), then the encounter takes place in close proximity with probability equal to $1-\sigma$; while if $k$ is susceptible ($X_k(t^-)=\text{S}$), then a physical encounter in close proximity always takes place. If an individual $j$ who is infected and mildly symptomatic ($X_j(t^-)=\text{I}$) activates and selects a susceptible individual $k$ ($X_k(t^-)=\text{S}$), then the encounter occurs in close proximity with probability $1-\sigma$; if $j$ selects another mildly symptomatic individual $k$  ($X_k(t^-)=\text{I}$), then they interact in close proximity with probability $(1-\sigma)^2$ (i.e. if both the individuals disregard the symptoms). We assume that quarantined individuals ($X_k(t^-)=\text{Q}$) do not participate in any (risky) encounters. If a physical encounter in close proximity occurs, then the ephemeral edge $(j,k)$ is added to the edge set $\mathcal{E}(t)$, and immediately removed afterwards.

The health state of an individual $j \in \mathcal{V}$ evolves over time according to the following two mechanisms.
\begin{description}
\item [Contagion.] Infection transmission occurs through pairwise physical encounters at close proximity. Specifically, if a susceptible individual $j$ $(X_j(t^-) =  \text{S})$ has a physical encounter with an undetected infected individual $k$ ($X_k(t^-) =  \text{I})$ and $(j,k)\in\mathcal E(t)$, the infection is transmitted to $j$ with \textit{per-contact infection probability} $\lambda \in [0,1]$. We assume that physical contact with detected individuals does not lead to transmission of the infection, since they will use protections, as they are aware of the risk. If infected, individual $j$ will develop severe symptoms $(X_j(t^+) =  \text{Q})$ with probability $p_{\text{q}}\in [0,1]$, while $j$ will be mildly symptomatic $(X_j(t^+) = \text{I})$ with probability  $1-p_{\text{q}}$. 
\item [Recovery.] An infected individual  $(X_j(t^-) \in \{ \text{I},  \text{Q} \})$ spontaneously recovers   according to a Poisson clock with rate $\beta \in \mathbb{R}_{>0}$, becoming again susceptible to the disease $(X_j(t^+) =  \text{S})$.
\end{description}

\subsection{Control measures}\label{sec:control}

\noindent Within our modelling framework, we introduce three control measures that public health authorities can take in order to control the spread of an epidemic disease.

\begin{description}
\item[Vaccination.]  Vaccination reduces the probability of infection transmission and the probability of the development of severe symptoms. To model these effects, we introduce the parameters $\gamma_{\text{t}} \in [0,1]$ and  $\gamma_{\mathrm{q}} \in [0,1]$, respectively. Let $v \in [0,1]$ denote the vaccination coverage of the population, i.e.\ the probability that a generic individual is vaccinated. The effect of vaccination is implemented in the model in a mean-field fashion, by multiplying the per-contact infection probability  $\lambda$ and the probability of developing severe symptoms $p_{\text{q}}$ by the re-scaling factors $(1-\gamma_{\text{t}} v)$  and $(1-\gamma_{\mathrm{q}} v)$, respectively.
\item[NPIs.] To quantify the effect of NPIs like mandatory face masks and enforcement of physical distancing, we introduce a parameter $\eta\in[0,1]$ that models the effectiveness of the measures implemented in preventing transmission. Hence, their effect is modelled by multiplying the per-contact infection probability  $\lambda$ by the quantity $(1-\eta)$.


\item[Testing.] By offering free testing campaigns, individuals with mild symptoms $(X_j(t) =  \text{I})$ are triggered to get tested. To quantify the effect of this measure, we introduce a Poisson clock with rate $c_{\text{t}}\in \mathbb{R}_{\ge 0} $, representing the rate at which testing takes place. Thus, mildly symptomatic infectious individuals $(X_j(t^-) =  \text{I})$ are diagnosed  according to a Poisson clock with rate $c_{\text{t}}$. After diagnosis, the infectious individuals quarantine themselves $(X_j(t^+) =  \text{Q})$, whereby they avoid having (risky) encounters, until they recover. 
\end{description}

\noindent All the parameters of the model are summarised in Table~\ref{tab:parameter}.

\section{Dynamics}

\begin{table}
\begin{minipage}{0.5\linewidth}
		\centering
		\label{model}
		\includegraphics[width=0.63\textwidth]{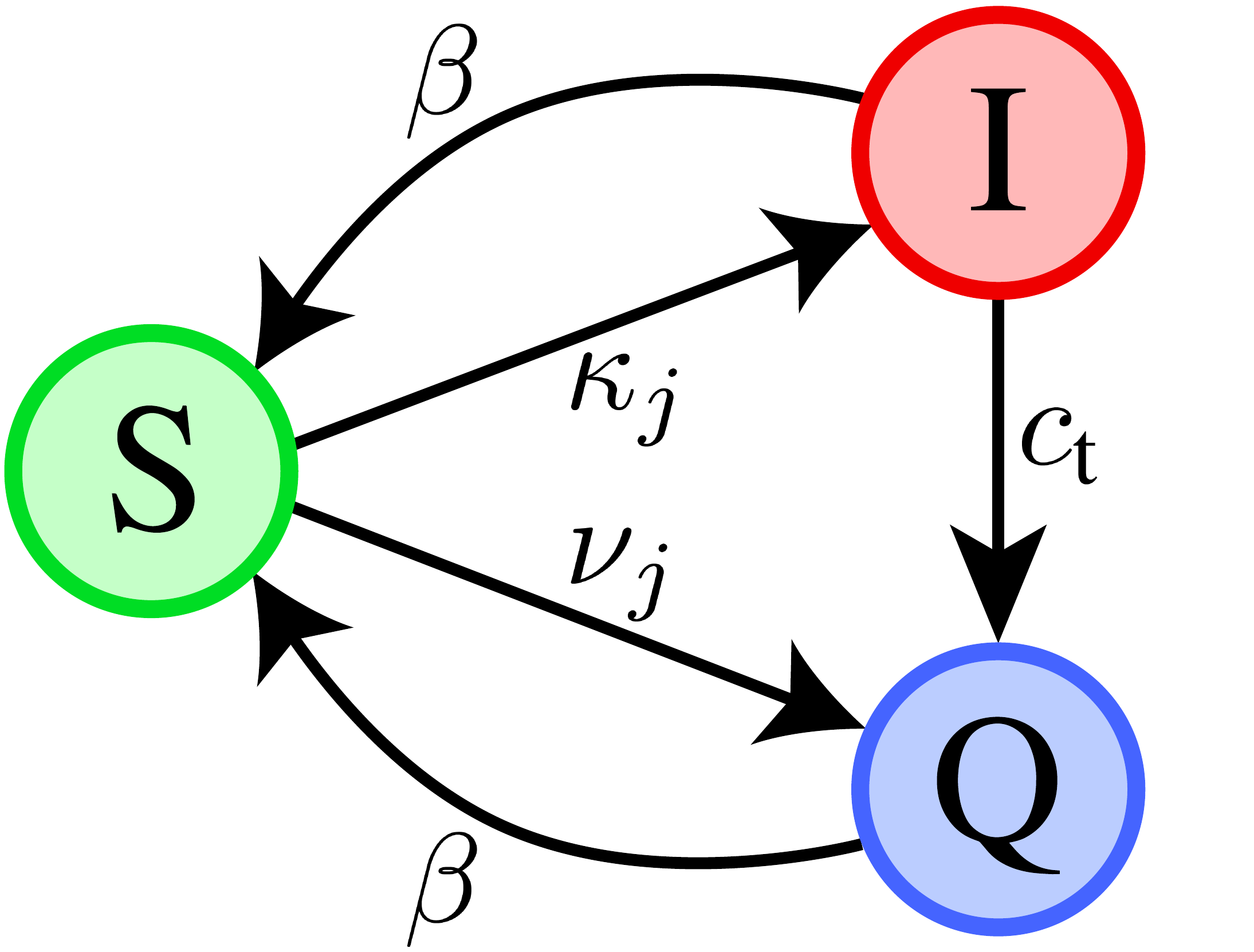}
		\captionof{figure}{State transitions of the epidemic model for $j\in \mathcal{V}$.}
	\end{minipage}	\hfill \begin{minipage}{0.5\linewidth}
		\caption{Model and Control Parameters.}
		\label{tab:parameter}
		\centering
		\resizebox{\textwidth}{!}{%
		\begin{tabular}{r| l}
$n\in\mathbb N$& population size\\
$t\in\mathbb R_{\geq 0}$& time\\
$X_j(t)\in\{\mathrm{S},\mathrm{I},\mathrm{Q}\}$& state of individual $j\in\mathcal V$ at time $t$\\
$\sigma\in[0,1]$&responsibility level of individuals\\
$\lambda\in[0,1]$& per-contact infection probability\\
$p_{\text{q}}\in[0,1]$& probability of severe illness\\
$\beta\in\mathbb R_{>0}$& recovery rate\\
$v\in[0,1]$& vaccination coverage\\
$\gamma_{\text{t}}\in[0,1]$&effectiveness of vaccine against transmission\\
$\gamma_{\mathrm{q}}\in[0,1]$&effectiveness of vaccine against severe illness\\
$\eta\in[0,1]$&effectiveness of NPIs\\
$c_{\text{t}}\in\mathbb R_{\geq 0}$&testing rate
\end{tabular}}
	\end{minipage}
\end{table}
\noindent All the mechanisms described in Section \ref{Epidemicmodel} are prompted by Poisson clocks, where each clock ticks independently of the others, so the $n$-dimensional vector $X(t) : = \left( X_1(t) \ X_2(t)  \cdots  X_n(t) \right)\in\{\text{S}, \text{I}, \text{Q}\}^{n}$, which represents the state of the population at time $t\in\mathbb R_{\geq 0}$, evolves according to a continuous-time Markov process~\cite{levin2006book}. As shown in Fig. 1, a generic individual can undergo five different state transitions, governed by the mechanisms of contagion, recovery, and testing, as described above. The three transitions triggered by recovery and testing are spontaneous processes, which occur with rate $\beta$ and $c_{\text{t}}$, as detailed in Sections \ref{Epidemicmodel} and \ref{sec:control}, respectively. The two triggered by contagion, instead, involve interactions between individuals and the corresponding rates are derived as follows.

If $X_j(t^-) = \text{S}$, then $j$ can become infected if they interact in close proximity with an infectious mildly symptomatic individual ($\text{I}$). Such an interaction occurs with rate equal to $2\frac{1-\sigma}{n-1}\cdot|\{k\in\mathcal V:X_k(t)=\text{I}\}|$, where the term $2$ comes from the fact that both $j$ and $k$ can initiate the encounter. When such an interaction occurs, $j$ becomes infected with probability equal to the per-contact infection probability, reduced by the presence of NPIs  and the effectiveness of vaccination against transmission. Then, $j$ becomes severely symptomatic with probability equal to $p_{\text{q}}(1-\gamma_{\mathrm{q}} v)$, otherwise they become infectious. Hence, we conclude that $X_j(t^-) = \text{S}$ becomes infectious ($X_j(t^+) = \text{I}$) according to a Poisson clock with rate \begin{align}
   \kappa_{j} := & 2 (1- \eta )\lambda\left(1-\gamma_{\text{t}} v\right)\left[1-p_{\text{q}}\left(1-\gamma_{\mathrm{q}} v \right)\right]  \dfrac{1-\sigma}{n-1}\sum\limits_{k\in\mathcal V:X_k= \text{I}  } 1,
\end{align}
while they become quarantined ($X_j(t^+) = \text{Q}$) according to a Poisson clock with rate 
 \begin{align}
   \nu_{j} := & 2 (1- \eta )\lambda\left(1-\gamma_{\text{t}} v\right)p_{\text{q}}\left(1-\gamma_{\mathrm{q}} v\right) \dfrac{1-\sigma}{n-1}\sum\limits_{k\in\mathcal V:X_k= \text{I}  } 1.
\end{align}

 
\noindent We begin our analysis by noting that the first row of the transition rate matrix $Q_j$ depends on the states of the other individuals. Hence, the individual dynamics cannot be decoupled, complicating the analysis for large-scale populations, where the state space dimension grows exponentially with $n$.  Thus, as is common practice in the study of these complex systems \cite{virusspread,9089218}, we employ a continuous-state deterministic mean-field relaxation of the system. In particular, instead of studying the actual evolution of the health state for each individual $j\in\mathcal V$, we study the evolution of the probabilities that the individual is susceptible, infectious, or quarantined, denoted as
\begin{equation*}
    s_{j} (t) : = \mathbb{P}\left[ X_j(t) = \text{S} \right], \quad
    i_{j} (t) : = \mathbb{P}\left[ X_j(t) = \text{I} \right], \quad \text{and} \quad
    q_{j} (t) : = \mathbb{P}\left[ X_j(t) = \text{Q} \right],
\end{equation*}
respectively. In the mean-field approximation, the evolution of these probabilities is governed by the system of ordinary differential equations $\left( \dot{s}_{j}  \  \dot{i}_{j} \  \dot{q}_{j} \right) = \left( s_{j} \  i_{j} \  q_{j} \right) Q_j$, or equivalently
\begin{align}\label{diffeq}
    \dot{s}_{j}  = &  - 2 s_{j} (1- \eta )\lambda\left(1-\gamma_{\text{t}} v\right)  \dfrac{1-\sigma}{n-1} \sum_{k \in \mathcal{V} \setminus \{j\} } i_{k}
      + \beta i_{j} + \beta q_{j}, \notag\\
    \dot{i}_{j}  = &  2 s_{j} (1- \eta )\lambda\left(1-\gamma_{\text{t}} v\right) \left[1-p_{\text{q}}\left(1-\gamma_{\mathrm{q}} v \right)\right]  \dfrac{1-\sigma}{n-1} \sum_{k \in \mathcal{V} \setminus \{j\} } i_{k} - (\beta + c_{\text{t}})  i_{j},  \\
  \dot{q}_{j}  = &  2 s_{j} (1- \eta )\lambda\left(1-\gamma_{\text{t}} v\right)p_{\text{q}}\left(1-\gamma_{\mathrm{q}} v \right)  \dfrac{1-\sigma}{n-1} \sum_{k \in \mathcal{V} \setminus \{j\} } i_{k} + c_{\text{t}} i_{j} - \beta  q_{j},   \notag
\end{align}
for all $j\in\mathcal V$.

\noindent Before moving to the actual mean-field analysis of the system, we define
\begin{equation}\label{propapprox}
  y_{\text{s}} := \dfrac{1}{n} \sum_{j \in \mathcal{V}} s_{j}, \quad y_{\text{i}} : = \dfrac{1}{n} \sum_{j \in \mathcal{V}} i_{j}, \quad y_{\text{q}} := \dfrac{1}{n} \sum_{j \in \mathcal{V}} q_{j},
\end{equation}
which is the average probability for a randomly selected individual to be in state $\text{S}$, $\text{I}$, and $\text{Q}$, respectively. For a sufficiently large enough $n$ and for any finite time-horizon, the fraction of individuals in a certain state can be arbitrarily closely approximated by \eqref{propapprox}, i.e.\ $ S(t) :=  \tfrac{1}{n}| \{j \in \mathcal{V}  :  X_j(t) =  \text{S}\} | \approx y_{\text{s}}$,  $I(t) :=  \tfrac{1}{n}| \{j \in \mathcal{V}  :   X_j(t) = \text{I}\} |\approx y_{\text{i}}$, and $
     Q(t) :=  \tfrac{1}{n}| \{j \in \mathcal{V}  :   X_j(t) =  \text{Q}\} | \approx y_{\text{q}}$ \cite{limittheo,zino2}, as illustrated in Fig. \ref{sim}. This supports the use of the mean-field approximation to study the emergent behaviour of the epidemic model for large populations employing the average probabilities defined in \eqref{propapprox} and the dynamics in \eqref{diffeq}. 

\begin{figure*}
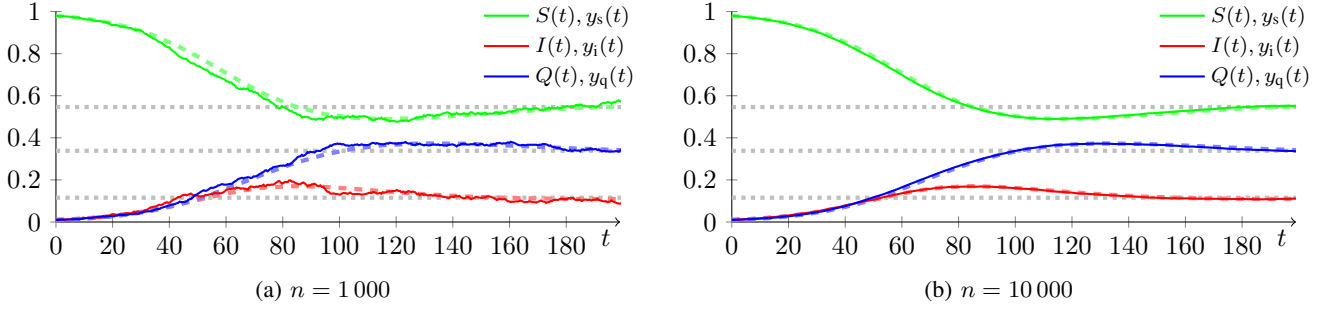
\vspace{4pt}
    \centering
\subfloat[$n=1\,000$]{\input{Figures/1k}}\quad\subfloat[$n=10\,000$]{\input{Figures/10k}}\\
    \caption{Simulations of the Markov process in \eqref{Q} (solid curves) and its deterministic mean-field approximation in \eqref{diffeq} (dashed curves) at the population level for increasing population sizes. The grey horizontal dotted lines represent the endemic equilibrium, computed in Theorem~\ref{theo2}. Common parameters are $\sigma=0.4$, $\lambda=0.2$, $p_{\mathrm{q}}=0.2$, $\beta=0.02$, $v=0.5$, $\gamma_{\mathrm{t}}=0.5$, $\gamma_{\mathrm{q}}=0.9$, $\eta=0.2$, and $c_{\mathrm{t}}=0.05$. }
    \label{sim}
\end{figure*}
\noindent For a generic $j^{\text{th}}$ entry of the Markov process $X(t)$, the transition rate matrix is given by
\begin{equation}\label{Q}
    Q_j  = \begin{bmatrix}
    -\kappa_{j}-\nu_{j} &   \kappa_{j}  &  \nu_{j}\\
    \beta & -\beta-c_{\text{t}}&  c_{\text{t}} \\
    \beta & 0 & -\beta
    \end{bmatrix},
\end{equation}
where the rows (columns) correspond to state $\text{S},  \text{I}$, and $ \text{Q}$, respectively. For any $p,q \in \{\text{S}, \text{I}, \text{Q}\}$ with $p \neq q$, 
\begin{equation*}
    \lim_{\Delta t\to 0}\frac{\mathbb{P}\left[X_j(t+ \Delta t) = q \,|\, X_j(t) = p \right]}{\Delta t} = (Q_j)_{pq} .
\end{equation*}


\section{Main Results}\label{section:mainresults}

\noindent In this section, we rigorously analyse the system in \eqref{diffeq}, to elucidate the role of vaccination and the level of responsibility in the epidemic spreading. 

Before presenting the main results, we will show that the system in \eqref{diffeq} is well-defined, i.e.\ that $\left( s_{j} \  i_{j} \  q_{j} \right)$ is a probability vector for all $j \in \mathcal{V}$ and $t \in \mathbb{R}_{\ge 0}$. 
\begin{lemma}\label{lemma1} For all $j \in \mathcal{V}$, the set
$\{ \left( s_{j}  \  i_{j} \  q_{j} \right) : s_{j}, i_{j}, q_{j} \ge 0, s_{j}+  i_{j}+ q_{j} = 1 \} $ is positive invariant under \eqref{diffeq}.
\end{lemma}
\begin{proof}
Observe that if one of the probabilities governed by \eqref{diffeq} is equal to zero, then its respective time-derivative is always non-negative. Furthermore, $ \dot{s}_{j}  +  \dot{i}_{j} +  \dot{q}_{j} =0 $, for all $j \in \mathcal{V}$, so $s_{j}+  i_{j}+ q_{j} = 1$ for all $t \in \mathbb{R}_{\ge 0}$. 
\end{proof}
\noindent  Note that Lemma~\ref{lemma1} also implies that only $2n$ of the $3n$ differential equations are linearly independent, as $s_j(t)+i_j(t)+q_j(t)=1$, for all $t\in\mathbb R_{\geq 0}$ and $j \in \mathcal{V}$, simplifying the analysis of the system.

The first thing we want to investigate is whether a local outbreak of the infection will escalate into a pandemic. Theorem \ref{theo1} presents the required conditions for (local) asymptotic stability of the disease-free equilibrium of system \eqref{diffeq}, which is coined as the \textit{epidemic threshold}. The threshold is presented as a critical value for the rate of testing $ c_{\mathrm{t}}$. If the testing rate is larger than $\bar{c}_{\mathrm{t}}$, as defined below, then the local outbreak is quickly extinguished; if not, it becomes endemic. 
     

\begin{theorem} \label{theo1} Consider the dynamical system in \eqref{diffeq}. Then, in the thermodynamic limit of large scale systems $n\to \infty$, the epidemic threshold is equal to
\begin{equation}\label{eq:threshold}
       \bar{c}_{\mathrm{t}} :=   2  (1- \eta )(1-\sigma)\lambda\left(1-\gamma_{\mathrm{t}} v\right) \left[1-p_{\mathrm{q}}\left(1-\gamma_{\mathrm{q}} v \right)\right] - \beta.
\end{equation}
In particular, if $ c_{\mathrm{t}} > \bar{c}_{\mathrm{t}}$, the disease-free equilibrium (with $y_{ \mathrm{i}}=y_{ \mathrm{q}}=0$) is locally asymptotically stable.
\end{theorem}
\begin{proof}
First, we immediately verify that the disease-free state of \eqref{diffeq}, that is,  $(s_j,i_{j},q_{j}) = (1,0,0)$ for all $j\in \mathcal{V}$, is always an equilibrium of the dynamics, as it nullifies the right-hand side of \eqref{diffeq}. To study its local stability, we consider a system made of the three macroscopic variables $y_{ \mathrm{s}}$, $y_{ \mathrm{i}}$, $y_{ \mathrm{q}}$, where its dynamics can be derived from \eqref{diffeq} and \eqref{propapprox} as
    \begin{align}\label{y}
    \dot{y}_{\mathrm{s}} = & - 2 (1- \eta )\lambda\left(1-\gamma_{\text{t}} v\right) (1-\sigma) y_{\mathrm{i}}  y_{\mathrm{s}}   + \beta y_{\mathrm{i}} + \beta y_{\mathrm{q}},\notag\\
    \dot{y}_{\mathrm{i}} = & 2 (1- \eta )\lambda\left(1-\gamma_{\text{t}} v\right) \left[1-p_{\text{q}}\left(1-\gamma_{\mathrm{q}} v \right)\right] (1-\sigma) y_{\mathrm{i}}  y_{\mathrm{s}} - (\beta+c_{\text{t}})y_{\mathrm{i}}, \\
    \dot{y}_{\mathrm{q}} = & 2 (1- \eta )\lambda\left(1-\gamma_{\text{t}} v\right) p_{\text{q}}\left(1-\gamma_{\mathrm{q}} v \right)  (1-\sigma) y_{\mathrm{i}}  y_{\mathrm{s}} +c_{\text{t}} y_{\mathrm{i}}-\beta y_{\mathrm{q}}.\notag
    \end{align}

\noindent Since $y_{\mathrm{s}}+y_{\mathrm{i}}+y_{\mathrm{q}} =1$ (as a consequence of Lemma~\ref{lemma1}), the system in \eqref{y} can be reduced to a planar system. For our analysis,  we take the second and third equation of \eqref{y}. Due the definition of the macroscopic variables, we observe that the disease-free equilibrium of \eqref{diffeq} is asymptotically stable if and only if the origin is asymptotically stable for the planar system $(y_{ \mathrm{i}},y_{ \mathrm{q}})$. We subsequently linearize \eqref{y} around the origin in the limit of large scale systems, $n \to \infty$, and obtain the following differential equations:
\begin{equation}\label{eq:macro}
    \begin{alignedat}{1}
  \dot{y}_{ \mathrm{i}} = &  2  (1- \eta )\lambda \left(1-\gamma_{\text{t}} v\right)\left[1-p_{\text{q}}\left(1-\gamma_{\mathrm{q}} v \right)\right](1-\sigma)y_{\text{i}}  - (\beta + c_{\text{t}})  y_{\text{i}},\\
    \dot{y}_{ \mathrm{q}} = & \left[ 2  (1- \eta )\lambda\left(1-\gamma_{\text{t}} v\right)p_{\text{q}}\left(1-\gamma_{\mathrm{q}} v \right)(1-\sigma)  +c_{\mathrm{t}} \right] y_{\text{i}}  - \beta   y_{\text{q}}.
\end{alignedat}
\end{equation}
The eigenvalues of the Jacobian matrix are given by $-\beta<0$, and 
\begin{equation}
       2  (1- \eta )(1-\sigma)\lambda\left(1-\gamma_{\mathrm{t}} v\right) \left[1-p_{\mathrm{q}}\left(1-\gamma_{\mathrm{q}} v \right)\right] - \beta- c_{\mathrm{t}},
\end{equation}
where the latter is negative if and only if $ c_{\mathrm{t}} > \bar{c}_{\mathrm{t}}$.
\end{proof}

 
 
\noindent  This theoretical result provides a rigorous tool to shed light on how the effectiveness of the vaccine and human behaviour impact the epidemic threshold of an infection, and thereby determine whether an epidemic outbreak could be easily controlled, or if higher testing efforts or more severe NPIs should be implemented. This will be discussed in the following remarks, and illustrated in the next section. 

\begin{remark}\label{rem:control} From the expression of the epidemic threshold $\bar c_{\mathrm{t}}$ in Theorem~\ref{theo1}, we observe that $ \frac{\partial\bar c_{\mathrm{t}}}{\partial \gamma_{\mathrm{t}}}<0$ and $\frac{\partial\bar c_{\mathrm{t}}}{\partial \gamma_{\mathrm{q}}}>0$. Hence, while high vaccine effectiveness against transmission favours the control of an epidemic outbreak, high effectiveness against severe illness hinders complete eradication. We can compute
\begin{equation}\label{eq:derivative1}
    \frac{\partial\bar c_{\mathrm{t}}}{\partial v}=2(1-\eta)(1-\sigma)\lambda(p_{\mathrm{q}}\gamma_{\mathrm{q}}-\gamma_{\mathrm{t}}(1-p_{\mathrm{q}})-2p_{\mathrm{q}}\gamma_{\mathrm{t}}\gamma_{\mathrm{q}}v),
\end{equation}
and find that whether an increase in the vaccination coverage facilitates complete eradication of an outbreak is non-trivial and depends on the probability of developing severe illness ($p_{\mathrm{q}}$), on the characteristics of the vaccine ($\gamma_{\mathrm{t}}$ and $\gamma_{\mathrm{q}}$), and on the fraction of people already vaccinated. Interestingly, at the beginning of the vaccination campaign, only vaccines with $\gamma_{\mathrm{t}}/\gamma_{\mathrm{q}}>p_{\mathrm{q}}/(1-p_{\mathrm{q}})$ favour the complete eradication of the disease. 
\end{remark}

\begin{remark}\label{rem:threshold}
Note that, in Theorem~\ref{theo1}, we have expressed the threshold in terms of the control effort that should be placed in testing in order to eradicate the outbreak. The same expression could be re-written, in order to highlight what the minimum level of individuals' responsibility or effectiveness of NPIs should be to obtain such a goal, as
\begin{align*}
       \bar\sigma&=1-  \frac{\beta+\bar{c}_{\mathrm{t}}}{2(1-\eta)\lambda\left(1-\gamma_{\mathrm{t}} v\right) \left[1-p_{\mathrm{q}}\left(1-\gamma_{\mathrm{q}} v \right)\right]},\\
 \bar\eta&=1-  \frac{\beta+\bar{c}_{\mathrm{t}}}{2(1-\sigma)\lambda\left(1-\gamma_{\mathrm{t}} v\right) \left[1-p_{\mathrm{q}}\left(1-\gamma_{\mathrm{q}} v \right)\right]},
 \end{align*}
respectively.
\end{remark}

\noindent We conclude this section by proving the following theorem, which characterises the global behaviour of the dynamical system \eqref{diffeq}. The theorem also provides an analytical expression of the endemic equilibrium of the system, when the disease-free equilibrium is unstable. The simulations in Fig.~\ref{sim} confirm our theoretical findings.

\begin{theorem} \label{theo2} Consider the dynamical system in \eqref{diffeq}, in the thermodynamic limit of large scale systems $n\to \infty$. If $ c_{\mathrm{t}} \geq \bar{c}_{\mathrm{t}}$, the system converges to the disease-free equilibrium $(y_{\mathrm{s}},y_{\mathrm{i}},y_{\mathrm{q}}) = (1,0,0)$. If $ c_{\mathrm{t}} < \bar{c}_{\mathrm{t}}$ and $y_{\mathrm{i}}(0)>0$, then the system converges to 
\begin{align}\label{equihoi}
 \Scale[1]{ y_{\mathrm{s}}^{\ast} = } & \Scale[1]{\dfrac{\beta + c_{\mathrm{t}}}{2(1- \eta )(1-\sigma)\lambda\left(1-\gamma_{\mathrm{t}} v\right)\left[1-p_{\mathrm{q}}\left(1-\gamma_{\mathrm{q}} v \right)\right] }},\notag \\
 \Scale[1]{  y_{\mathrm{i}}^{\ast} = } & \Scale[1]{\dfrac{\beta \left[1-p_{\mathrm{q}}\left(1-\gamma_{\mathrm{q}} v \right)\right]}{\beta + c_{\mathrm{t}}} - \dfrac{\beta}{2(1- \eta )(1-\sigma)\lambda\left(1-\gamma_{\mathrm{t}} v\right)}},\\
 y_{\mathrm{q}}^{\ast} = & \left(1 - \dfrac{\beta\left[1-p_{\mathrm{q}}\left(1-\gamma_{\mathrm{q}} v \right)\right]}{\beta + c_{\mathrm{t}}} \right) \left[1- \dfrac{\beta + c_{\mathrm{t}}}{2(1- \eta )(1-\sigma)\lambda\left(1-\gamma_{\mathrm{t}} v\right) \left[1-p_{\mathrm{q}}\left(1-\gamma_{\mathrm{q}} v \right)\right]} \right]. \notag
 \end{align}

\end{theorem}

\begin{proof}
 For this analysis, we use again the fact that $y_{\mathrm{s}}+y_{\mathrm{i}}+y_{\mathrm{q}}=1$ (Lemma~\ref{lemma1}) and we take the first and second equation of \eqref{y}. Note that the domain is bounded (Lemma \ref{lemma1}), so we can employ the Bendixson-Dulac theorem \cite{ricardo2020modern}. For $\phi(y_{\mathrm{s}}, y_{\mathrm{i}})=(y_{\mathrm{s}}y_{\mathrm{i}})^{-1}$ we find
\begin{equation}
    \dfrac{\partial (\phi \dot{y}_{\textrm{s}})}{\partial y_{\textrm{s}}} +  \dfrac{\partial (\phi \dot{y}_{\textrm{i}})}{\partial y_{\textrm{i}}} = - \dfrac{\beta}{y_{\textrm{s}}^2} \left(1+ \dfrac{1- y_{\textrm{i}}}{ y_{\textrm{i}}} \right) <0, 
\end{equation}

\noindent so there do not exist periodic solutions of \eqref{y}. If $c_{\mathrm{t}} > \bar{c}_{\mathrm{t}}$, then the disease-free equilibrium $(y_{\mathrm{s}},y_{\mathrm{i}},y_{\mathrm{q}}) = (1,0,0)$ is the unique equilibrium of \eqref{y} and therefore globally asymptotically stable. If $c_{\mathrm{t}} < \bar{c}_{\mathrm{t}}$, then one can easily verify that there exists one other equilibrium given by \eqref{equihoi}, which we term endemic equilibrium. Since there do not exist periodic solutions and the disease-free equilibrium is unstable for $c_{\mathrm{t}} < \bar{c}_{\mathrm{t}}$, the endemic equilibrium is asymptotically stable. Note that for $c_{\mathrm{t}} = \bar{c}_{\mathrm{t}}$, \eqref{equihoi} is equal to the disease-free equilibrium, hence, we conclude that for $c_{\mathrm{t}} = \bar{c}_{\mathrm{t}}$, the disease-free equilibrium is asymptotically stable.
\end{proof}

\begin{remark}\label{rem:mitigate}
Note that in the endemic equilibrium \eqref{equihoi}, the fraction of seriously ill individuals in the population can be computed as
\begin{equation}
\xi: =(1-y_{\mathrm{s}}^*)p_{\mathrm{q}}\left(1-\gamma_{\mathrm{q}} v \right).
\end{equation}
Using the explicit expression of $y_{\mathrm{s}}^*$, we can compute
   \begin{align*} \displaystyle\frac{\partial \xi}{\partial v}=&\displaystyle-\frac{\partial y_{\mathrm{s}}^*}{\partial v}p_{\mathrm{q}}(1-\gamma_{\mathrm{q}} v)-(1-y_{\mathrm{s}}^*)p_{\mathrm{q}}\gamma_{\mathrm{q}}\\
   =&\displaystyle\frac{\beta+c_{\mathrm{t}}}{(\beta+\bar c_{\mathrm{t}})^2}\frac{\partial \bar c_{\mathrm{t}}}{\partial v}p_{\mathrm{q}}(1-\gamma_{\mathrm{q}} v)-\left(1-\frac{\beta+c_{\mathrm{t}}}{\beta+\bar c_{\mathrm{t}}}\right)p_{\mathrm{q}}\gamma_{\mathrm{q}},
   \end{align*}
where $\bar{c}_{\mathrm{t}}$ is defined in \eqref{eq:threshold} and $\frac{\partial \bar c_{\mathrm{t}}}{\partial v}$ is computed in \eqref{eq:derivative1}. Hence, vaccination seems to be beneficial in reducing the fraction of people with severe illness for the set of parameters that satisfy $\displaystyle\frac{\partial \xi}{\partial v}>0$, which includes all the cases in which the vaccine facilitates the control of the epidemic (Remark~\ref{rem:control}). 
\end{remark}

\noindent Our theoretical findings highlight that, for a wide range of parameters (Remark~\ref{rem:mitigate}), vaccination is a powerful control measure for the mitigation of an epidemic outbreak, as it is able to successfully reduce the number of seriously infected individuals in the population. However, Remark~\ref{rem:control} shows that, in some of these scenarios, vaccination could act as a double-edged sword, whereby, despite reducing severe infections, it could hinder complete eradication of the disease. In these scenarios,  the epidemic disease tends to become endemic, unless stronger control measures are enacted (NPIs or efficient testing practices) or people increase their responsibility level, restraining themselves from having interactions with others if they have symptoms, as indicated by Theorem~\ref{theo1} and Remark~\ref{rem:threshold}. In the next section, we will illustrate this non-trivial effect of vaccination by means of a case study, inspired by the ongoing COVID-19 pandemic and vaccination campaign.

 
\section{Case Study and Discussion}

\begin{figure*}\vspace{4pt}
\centering
\subfloat[Prevalence of severe illness $\xi$]{\label{fig:gamma1}\includegraphics[]{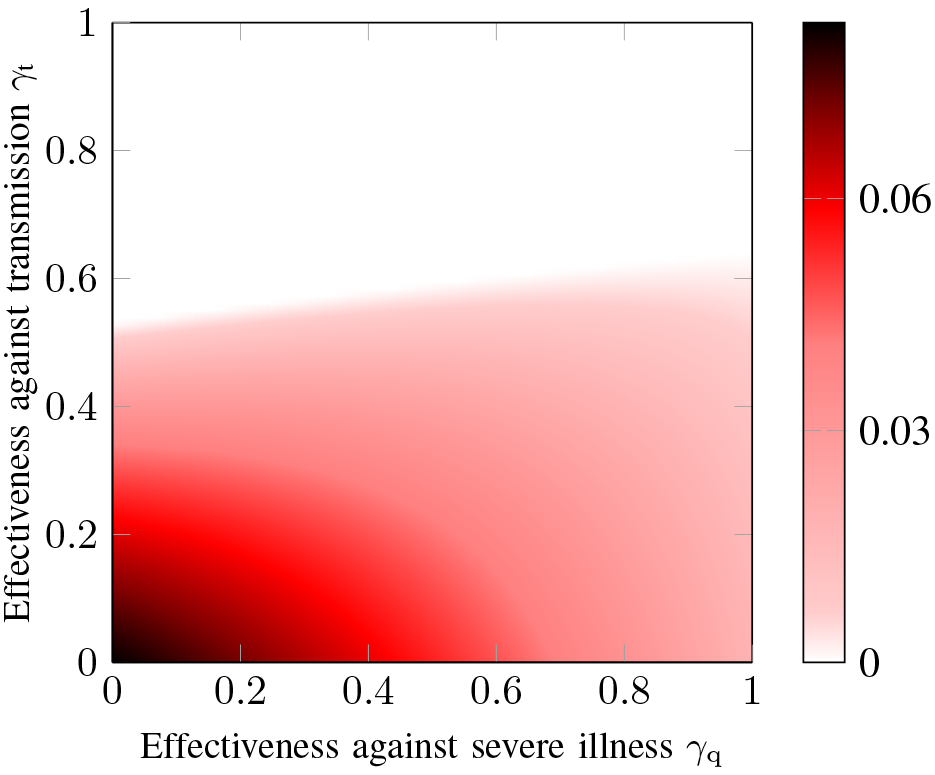}}\quad
\subfloat[Epidemic threshold $\bar{c}_{\mathrm{t}}$]{\label{fig:gamma2}\includegraphics[]{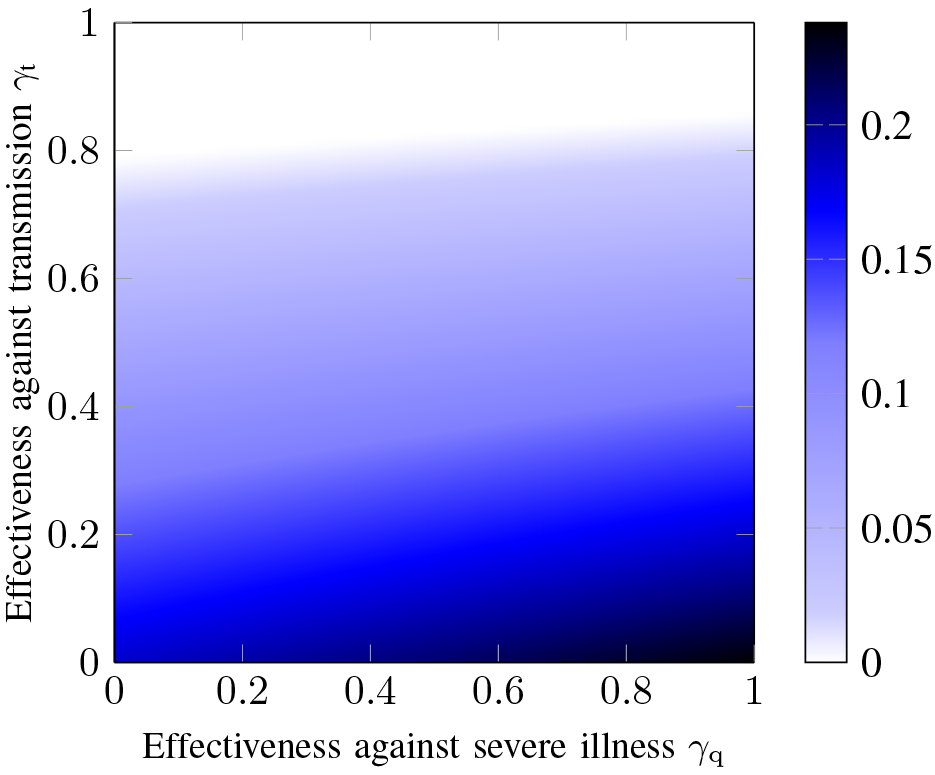}}
\caption{(a) Prevalence of severe illness $\xi$  and (b) epidemic threshold for different levels of the effectiveness against severe illness
$\gamma_{\mathrm{q}}$ and the effectiveness against transmission $\gamma_{\text{t}}$. Common parameters are $\sigma=0.4$, $v=0.821$, $\eta =0.19$, $c_{\text{t}}=0.06$ $\lambda = 0.36$, $p_{\text{q}}=0.19$, and $\beta=0.1$.}\label{fig:gamma}
\end{figure*}

\noindent Here, we present a case study to provide insights into the theoretical results presented in Section \ref{section:mainresults}. In particular, motivated by the ongoing COVID-19 pandemic and global vaccination campaign, we calibrate our model to reflect some characteristics of COVID-19 and of the situation in the Netherlands as of early November 2021. In \cite{phucharoen2020characteristics}, it was estimated that the infection probability of a contact is 36\%, which was reduced to 29\% after governmental NPIs. Thus, we take $\lambda=0.36$, and derive from $\lambda(1-\eta) = 0.29$ that $\eta=0.19$. According to clinical data, 81\% of COVID-19 cases are classified as mild \cite{ChineseCDC}, so we set $p_{\text{q}}=0.19$. Note that, for this choice of parameters, the vaccine seems to be almost always beneficial in mitigating the disease (by checking the condition in Remark~\ref{rem:mitigate}), but, depending on the characteristics $\gamma_{\text{t}}$ and $\gamma_{\text{q}}$ of the vaccine in question, it may hinder the possibility of fully eradicating the disease without increasing NPIs or testing, or relying on a more responsible population. Vaccine efficacy is dependent on the type of vaccine, the time passed from its administration, and the virus strain, and there is not yet a common consensus in the scientific community on reliable estimations for the parameters $\gamma_{\text{t}}$ and $\gamma_{\text{q}}$.

\noindent In view of these high levels of uncertainty, we use our theoretical findings to examine the effect of the vaccine effectiveness parameters $\gamma_{\text{t}}$ and $\gamma_{\text{q}}$ on the prevalence of severe illness (Fig. \ref{fig:gamma1}) and on the epidemic threshold (Fig. \ref{fig:gamma2}). From Fig. \ref{fig:gamma}, we observe that although a vaccination coverage of $82.1\%$ (that is, the coverage in the Netherlands as of November 5, 2021~\cite{Mathieu2021database,rivmvaccinecoverage}) has a strongly beneficial effect on the reduction of the prevalence of seriously ill individuals (Fig.~\ref{fig:gamma1}), the epidemic threshold is monotonically increasing in the effectiveness against severe illness $\gamma_{\mathrm{q}}$ (Fig.~\ref{fig:gamma2}). This implies that, unless testing is performed at a very high rate, the current vaccination coverage in the Netherlands as of early November 20201 is not enough to stop local outbreaks from becoming endemic and completely eradicate the disease. In other words, despite vaccines being highly effective in the prevention of severe illness, thereby reducing the pressure on hospitals, they may complicate the control of local outbreaks. This is a consequence of the increased social activity of infected individuals, who are less likely to develop severe symptoms, and thus are less likely to be detected and isolated.

Finally, we consider a specific scenario calibrated on the BNT162b2 mRNA vaccine (Comirnaty by Pfizer-BioNTech), which is the most used vaccine in the Netherlands~\cite{Mathieu2021database}. We consider a scenario in which we set $\gamma_{\text{t}}=0.65$, as \cite{de2021vaccine} indicated a reduction of the risk of transmission of 65\%  for Comirnaty. According to clinical studies~\cite{dagan2021bnt162b2}, a full vaccination status reduces the probability of severe symptom development with  $92\%$, so we set   $\gamma_{\text{q}}=0.92$. In Fig. \ref{fig4}, we show that the epidemic threshold and the fraction of seriously ill population are both monotonically decreasing in the vaccination coverage $v$ and the responsibility level $\sigma$. However, compared to the responsibility level, vaccination has a lower impact, and does not lead to a complete eradication of the disease at  low responsibility levels (Fig. \ref{fig:severe}), suggesting that responsibility is key to achieve eradication of the disease. Some recent clinical studies suggest that the efficacy of Comirnaty ---in particular against transmission--- may quickly wane over the course of a few months after the second shot, calling for the need of a booster shot campaign~\cite{Eyre2021}. While there is still not a consensus in the scientific community about waning immunity and its timing, we perform some further analysis by assuming that, after some months, the effectiveness against transmission reduces to a fourth of the nominal value, that is, $\gamma_{\text{t}}=0.165$.  Our findings suggest that although a high vaccination coverage is still effective in keeping the fraction of seriously ill population under control (Fig.~\ref{fig:severe2}), it is not beneficial in facilitating the complete eradication of the disease, since the degree of vaccination coverage has no effect on the epidemic threshold (Fig.~\ref{fig:threshold2}). 

 \begin{figure*}\vspace{4pt}
\centering
\subfloat[Seriously ill $\xi$]{\label{fig:severe}\includegraphics[]{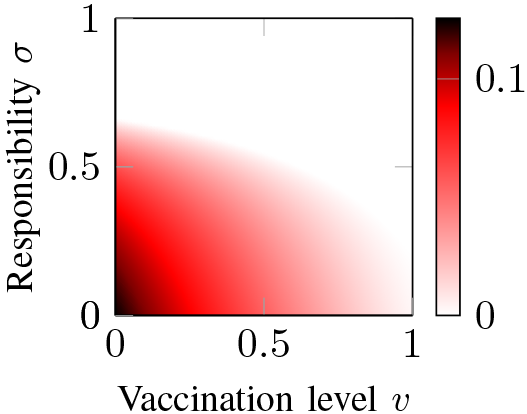}}\subfloat[Threshold $\bar{c}_{\mathrm{t}}$]{\label{fig:threshold}\includegraphics[]{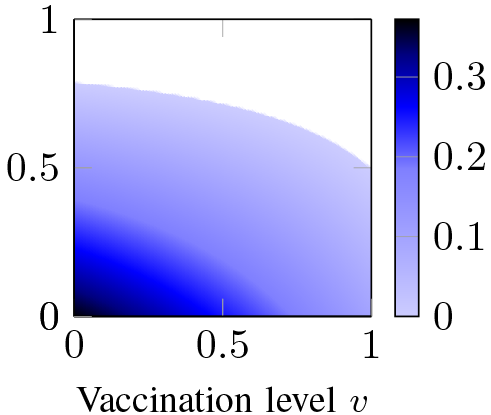}}\ \subfloat[Seriously ill $\xi$]{\label{fig:severe2}\includegraphics[]{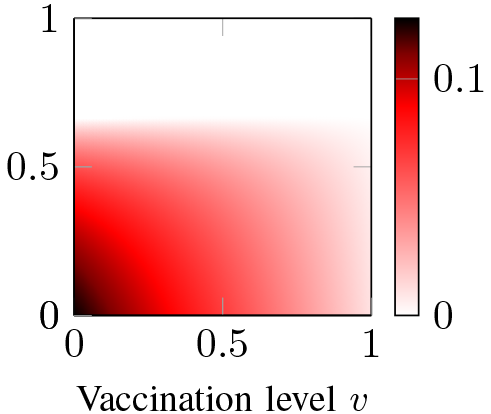}}\subfloat[Threshold $\bar{c}_{\mathrm{t}}$]{\label{fig:threshold2}\includegraphics[]{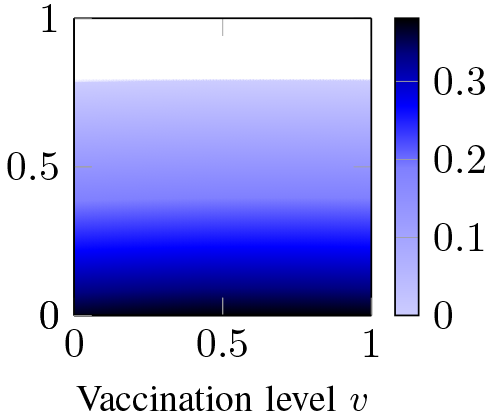}}
\caption{(a,c) Epidemic threshold and (b,d) fraction of seriously ill population $\xi$ for different levels of vaccination $v$ and responsibility $\sigma$. In (a-b), $\gamma_{\text{t}} = 0.65$; in (c-d), $\gamma_{\text{t}} = 0.165$. Common parameters are  $\gamma_{\mathrm{q}} = 0.92$, $\eta =0.19$, $\lambda = 0.36$, $p_{\text{q}}=0.19$, $c_{\text{t}}=0.06$, and $\beta=0.1$.
}\label{fig4}
\end{figure*}

Before concluding this paper, we would like to remark that our results are derived with a simple epidemic model that does not capture important features of COVID-19 like latency period and (temporary) immunisation after recovery. Hence, the case study presented here should be meant as a preliminary qualitative study of our theoretical results inspired by a real-world scenario and a motivating example to perform future research with more complex epidemic models, such as those typically used for COVID-19~\cite{giulia,parino2021,Carli2020,DellaRossa2020,Calafiore2020}, toward deriving rigorous quantitative conclusions.


\section{Conclusion}
\noindent We proposed a mathematical model for the spreading of recurrent epidemic infections in a population that is partially vaccinated, and studied the impact of vaccination campaigns on the epidemic prevalence and the  control of local outbreaks. Employing a mean-field approximation of the system, in the limit of large-scale populations, we derived the epidemic threshold and an expression for the endemic equilibrium. Our theoretical results indicate that although vaccines have a beneficial effect on the prevalence of individuals with severe symptoms, it may impede the control of local outbreaks. To manage such outbreaks, the combination of responsible behaviour of individuals and effective testing practices is key. In this paper, we kept the level of responsibility constant for all individuals.

The promising preliminary results presented in this paper pave the way for several avenues of future research. First, we plan to incorporate further compartments, to capture more realistic features of epidemics such as (temporary) immunity after recovery and latency periods, similar to the model proposed to study COVID-19~\cite{giulia,parino2021,Carli2020,DellaRossa2020,Calafiore2020}. Second, similar to~\cite{Kat,ye2021game}, we aim at encapsulating in the model a game-theoretic decision making process, whereby individuals make the decision to protect others based on a combination of external factors. As another direction, one may consider an extension of the proposed model by the implementation of a game-theoretic approach for the adoption of vaccination, as is done in \cite{fu2011imitation}.
Finally, one may consider interactions between multi-populations with a different vaccination coverage, and study the effect of a community with a very low vaccination rate (e.g.\ the Bible Belt in the Netherlands \cite{rivmvaccinecoverage}) on epidemic spreading.

\ifCLASSOPTIONcaptionsoff
  \newpage
\fi



%

\bibliographystyle{ieeetr}
\bibliography{bib}

%




\end{document}